\documentclass{article}
\usepackage{amsmath,amssymb,amsthm}
\title{Relatively Computably Enumerable Reals}
\author{Bernard A. Anderson}

\newtheorem{theorem}{Theorem}[section]

\newtheorem*{claim}{Claim}
\theoremstyle{definition}
\newtheorem{definition}{Definition}

\newcommand{\setsep}{\ensuremath{\, | \;}}
\newcommand{\reals}{2^\omega}
\newcommand{\strings}{2^{<\omega}}
\newcommand{\num}{\in\omega}

\newcommand{\cset}[1]{\ensuremath{\overline{#1}}}

\newcommand{\concat}{^\smallfrown}
\newcommand{\rstrd}{\mathrel{\mbox{\raisebox{.5mm}{$\upharpoonright $}}}} 
\newcommand{\smrstrd}{\mathrel{\mbox{\raisebox{.2mm}{$\scriptstyle \upharpoonright $}}}}

\newcommand{\es}{\emptyset}
\newcommand{\ce}{c.e.\ }

\begin{document}
\maketitle

\begin{abstract}
A real $X$ is defined to be relatively \ce if there is a real $Y$ such that $X$ is c.e.($Y$) and $X \not\leq_T Y$.  A real $X$ is 
relatively simple and above if there is a real $Y <_T X$ such that $X$ is c.e.($Y$) and there is no infinite set $Z \subseteq \cset{X}$ such 
that $Z$ is c.e.($Y$).  We prove that every nonempty $\Pi^0_1$ class contains a member which is not relatively c.e.\ and that every 
1-generic real is relatively simple and above.
\end{abstract}

\section{Introduction}

In this paper, we consider the following question: For which $X \in \reals$ is there a $Y$ such that $X$ is c.e.($Y$) but $X \not\leq_T Y$?  Such 
reals $X$ are said to be relatively computably enumerable.  Similarly, we can also look at which reals hold other basic properties (such as being 
simple) relative to another real.  We briefly summarize some known results and prove two new theorems on the topic.

We begin with some definitions.

\begin{definition}A real $X$ is relatively \ce if there is a real $Y$ such that $X$ is c.e.($Y$) and $X \not \leq_T Y$. \end{definition}
\begin{definition}A real $X$ is relatively c.e.a.\ (computably enumerable and above) if there is a $Y <_T X$ such that $X$ is c.e.($Y$).    \end{definition}
\begin{definition}A real $X$ is relatively s.a.\ (simple and above) if there is a real $Y <_T X$ such that $X$ is c.e.($Y$) and there is no 
infinite set $Z \subseteq \cset{X}$ such that $Z$ is c.e.($Y$).   \end{definition}
\begin{definition} For reals $A$ and $B$, we say $A \leq_e B$ ($A$ is enumeration below $B$) if there is a $\Sigma_1$ set $C$ such that 
$n \in A$ iff there is a finite $E \subseteq B$ (as a set) with $(n,E) \in C$.  \end{definition}
\begin{definition}A real $X$ is 1-generic if for every $\Sigma_1$ set $S \subseteq \strings$ either $X$ meets $S$ or for some $k \num$, 
every $\sigma \supseteq X \rstrd k$ is such that $\sigma \notin S$. \end{definition}

Jockusch \cite{gen} proved that every 1-generic real is relatively c.e.a., and Kurtz \cite{meas} showed that the set of relatively c.e.a.\ reals has 
measure one.  Kautz \cite{kau} improved this by demonstrating that all 2-random reals are relatively c.e.a.  

While the set of relatively c.e.\ reals seems very large, there is a natural limit to its size.  Given any real $X$, we can find a real $Y$ in the same truth-table degree such that $Y$ is not relatively c.e.  We simply let $Y$ code the initial segments of $X$\@.  More generally, a real $X$ is not relatively c.e.\ any time $\cset{X} \leq_e X$.  It follows from a result of Selman \cite{Selman} that this is the only case in which a real is not relatively c.e.  

\begin{theorem}[Selman \cite{Selman}] $X$ is relatively c.e.\ if and only if $\cset{X} \not\leq_e X$. \label{SelThm} \end{theorem}

Selman's result shows that the set of relatively c.e.\ reals is, in some sense, as large as possible.

Jockusch showed, in a paper by Case \cite{Case}, that every Turing degree contains a real $X$ such that $\cset{X} \not\leq_e X$, and 
hence by Theorem \ref{SelThm}, that every Turing degree contains a real which is relatively c.e.  If we apply this to a minimal degree, 
we have a real which is relatively c.e.\ but not relatively c.e.a.  For recent results on relatively c.e.a.\ reals, see Ambos-Spies, Ding, Wang, 
and Yu \cite{ceaResults}.

Jockusch and Soare \cite{hifran} showed that every nonempty $\Pi_1^0$ class contains a real which has hyperimmune-free degree and hence is not 
relatively c.e.a.  We sharpen this by showing every nonempty $\Pi_1^0$ class contains a real which is not relatively c.e.  
We also improve Jockusch's \cite{gen} result that every 1-generic real is relatively c.e.a., by proving that every 1-generic real is relatively s.a. 

This work formed part of the author's Ph.D.\ thesis at the University of California at Berkeley.  We thank Theodore Slaman, the 
dissertation supervisor, for his ideas on the topic and his repeated suggestions of new approaches to problems.

\section{$\Pi_1^0$ Classes}

\begin{theorem} Every nonempty $\Pi_1^0$ class contains a member which is not relatively c.e. \label{piclass} \end{theorem}
\begin{proof} Let $T$ be a computable tree such that the members of the $\Pi_1^0$ class are the paths through $T$.  We will inductively 
construct a real $X$ which will be the rightmost path through $T$ and a set $C$ which will witness $\cset{X} \leq_e X$.  The 
procedure will be computable and we will only add elements to $C$.  Hence $C$ will be c.e.

We begin with $X_0 = \langle \rangle$ and $C_0 = \emptyset$.  At stage $s+1$, if $X_s \concat 1 \in T$ we let $X_{s+1} = X_s \concat 1$ 
and $C_{s+1} = C_s$.  Otherwise, let $l$ be greatest such that $(X_s \rstrd l) \concat 0 \in T$ ($l$ must exist since the class is nonempty).  We 
then let $X_{s+1} = (X_s \rstrd l) \concat 0$ and $C_{s+1} = C_s \cup (l, \{ n < l \setsep X_s (n) = 1 \} )$.

We observe that $X$ is the rightmost path through $T$ (we set $X(m)=0$ if and only if there is no path through $T$ extending $(X \rstrd m) 
\concat 1$).  We wish to show for all $m \num$ that $m \notin X$ if and only if $(m,E) \in C$ for some finite $E \subset X$ (viewed as 
a set).  Suppose $m \notin X$.  Let $s$ be least such that for all $t>s$ we have $X_s \rstrd m = X_t \rstrd m$.  Then $(m, \{ n<m \setsep X(n)=1 \} )$ 
was added to $C$ at stage $s+1$.  

Conversely, suppose $E \subset X$ and $(m,E)$ was added to $C$ at stage $s$.  We note that the value of $X$ at $n$ does not change from 0 to 1 unless the value of $X$ at $k$ changes from 1 to 0 for some $k < n$.  As a result, if $X_s \rstrd m \neq X_t \rstrd m$ for some $t>s$, then 
$\{ n<m \setsep X_s(n)=1 \} \not\subseteq \{ n<m \setsep X(n)=1 \}$.  This implies $E \not\subseteq X$ for a contradiction.  We conclude $X_s \rstrd m = X \rstrd m$, and hence $m \notin X$.  

Therefore $C$ witnesses $\cset{X} \leq_e X$, so by Theorem \ref{SelThm} we have that $X$ is not relatively c.e.
\end{proof}

Given any property such that there is a nonempty $\Pi_1^0$ class of reals that have the property, we can apply Theorem \ref{piclass} to find a real with this property which is not relatively c.e.  For example, it provides alternate proofs that there are 1-random reals and Schnorr trivial reals that are not relatively c.e.\ (Franklin and Stephan \cite{Franklin} constructed a $\Pi_1^0$ class, all of whose members are Schnorr trivial).  Using the contrapositive of Theorem \ref{piclass}, we have alternate proofs that there is no $\Pi_1^0$ class of 1-generics or 2-randoms.

\section{Relatively Simple and Above}

\begin{theorem} Let $X$ be 1-generic.  Then $X$ is relatively s.a.  \end{theorem}
\begin{proof} Let $\langle , \rangle$ be a computable pairing function such that $\langle m,n \rangle > m$ for all $m, n$.  Let $Y = \{ \langle n,m \rangle \setsep n \in X \ \wedge\ \langle n,m \rangle \notin X \}$.  Then $Y \leq_T X$, and we have that $X$ is c.e.($Y$), since $n \in X$ if and only if $\exists m\: [ \langle n,m \rangle \in Y]$ (since $X$ is generic, we can't have $\langle n,m \rangle \in X$ for every $m$).  It remains to show that there is no infinite $Z \subseteq \cset{X}$ such that $Z$ is c.e.($Y$) [this also gives $X \not\leq_T Y$, since $X \leq_T Y$ implies $\cset{X}$ is c.e.($Y$)].

Suppose towards a contradiction there is an infinite $Z \subseteq \cset{X}$ such that $Z = W^Y_k$ for some $k$.  We define a function 
$j: \strings \to \strings$ such that $Y = j(X)$\@.  Let $[j(\sigma)](\langle n,m \rangle) = 1$ iff $\sigma(n) = 1$ and 
$\sigma (\langle n,m \rangle) = 0$.  

To prove the theorem, we will use the facts that $X \cap W^{j(X)}_k = \emptyset$ and $X$ is 1-generic.  This implies there is an initial segment $X \rstrd l$ such that for every extension $\tau \supseteq X \rstrd l$ we have $\tau \cap W^{j(\tau)}_k = \emptyset$.  We will then get a contradiction by adding an element of $W^{j(\tau)}_k$ to $\tau$ without changing $j(\tau)$.  

Let $S = \{ \sigma \setsep \exists n \: [n \in \sigma\ \wedge\ n \in W^{j(\sigma)}_k] \}$.  Then $X \notin S$, so let $l$ be such that 
for every $\tau$ extending $X \rstrd l$ we have $\tau \notin S$\@.  Since $Z$ is infinite, let $p \in Z$ with $p>l$, and let $t>l$ be such 
that $p \in W^{j(X \smrstrd t)}_k$.  We note that for any $\sigma \supseteq X \rstrd l$ such that $j(\sigma) = j(X \rstrd t)$, we have $p \in W^{j(\sigma)}_k$ 
and $\sigma \notin S$, so $p \notin \sigma$.  We can now obtain a contradiction.

\begin{claim} There is a $\sigma \supseteq X \rstrd l$ such that $j(\sigma) = j(X \rstrd t)$ and $p \in \sigma$. \end{claim}
\begin{proof} We define a sequence of strings $\sigma_i$ of length $t$ inductively.  Let $\sigma_0 = X \rstrd t$.  Let $\sigma_1$ be given by
$$\sigma_1 (n) = \begin{cases} 1 & \; n = p \\ \sigma_0 (n) & \; \mbox{else} \end{cases}$$
At each stage we will remove all witnesses of changes made in the previous stage.

For each stage $i \geq 2$ we define $\sigma_i$ by
$$\sigma_i(\langle b,a \rangle)= \begin{cases} 1 & \; \sigma_{i-1}(b) \not= \sigma_{i-2}(b) \\ \sigma_{i-1}(\langle b,a \rangle) & \; \mbox{else} \end{cases}$$
We note that since $\langle b,a \rangle > b$, the least $m$ such that $\sigma_i(m) \not= \sigma_{i-1}(m)$ strictly increases with $i$.  Hence for some stage $i$ we have $\sigma_i = \sigma_{i-1}$, and we let $\sigma$ be this $\sigma_i$.  We also note that this procedure never changes a 1 to a 0.

We have $p \in \sigma$ and note that $\sigma \supseteq X \rstrd l$ since $p>l$.  It remains to show that $j(\sigma) = j(X \rstrd t)$.

Let $n,m$ be arbitrary such that $[j(X \rstrd t)](\langle n,m \rangle)=1$.  Then $X(n)=1$ and $X(\langle n,m \rangle )=0$; so at every stage $i$, $\sigma_i(n)=1$ and $\sigma_i(\langle n,m \rangle)= \sigma_{i-1}(\langle n,m \rangle )$.  Hence $\sigma(n)=1$ and $\sigma(\langle n,m \rangle)=0$ so $[j(\sigma)](\langle n,m \rangle)=1$.  

Conversely, let $n,m$ be arbitrary such that $[j(\sigma)](\langle n,m \rangle)=1$.  Then $\sigma(n)=1$ and $\sigma(\langle n,m \rangle)=0$.  The latter implies $X(\langle n,m \rangle)=0$.  Suppose $X(n)=0$.  Let $i$ be least such that $\sigma_i(n)=1$.  Then $\sigma_{i+1}(\langle n,m \rangle)=1$, 
so $\sigma(\langle n,m \rangle)=1$ for a contradiction.  Hence $X(n)=1$ so $[j(X \rstrd t)](\langle n,m \rangle)=1$.

Therefore $j(\sigma) = j(X \rstrd t)$.
\end{proof}
Thus, there is no infinite $Z \subseteq \cset{X}$ which is c.e.($Y$).  Therefore $X$ is relatively s.a.
\end{proof}

\section{Conclusion}

There is still room to investigate relatively c.e.\ reals.  One question we can consider is: Is there a 1-random $X \not\geq_T \es^\prime$ such that $X$ is not relatively c.e.?  By Franklin and Ng \cite{DifRan} this is equivalent to the question: Is there a difference random which is not relatively c.e.?  The approach used in this paper cannot be used directly to answer this question.  The real constructed by applying Theorem \ref{piclass} to a $\Pi^0_1$ class of 1-randoms is Turing complete (as is Chaitin's $\Omega$).  Similarly, we can ask: Is there a weakly 2-random real which is not relatively c.e.?  We can also look at further classifying which reals are relatively s.a. 

\nocite{*}
\bibliography{RelativeCE}
\end{document}